\documentclass{amsart}
\usepackage{amssymb}
\usepackage{amsfonts}

\newtheorem{theorem}{Theorem}
\theoremstyle{plain}

\newtheorem{corollary}{Corollary}

\newtheorem{definition}{Definition}

\newtheorem{lemma}{Lemma}

\newtheorem{remark}{Remark}

\numberwithin{equation}{section}
\input{tcilatex}

\begin{document}
\title[Hermite-Hadamard type inequalities via preinvexity and
prequasiinvexity]{Hermite-Hadamard type inequalities via preinvexity and
prequasiinvexity}

\begin{abstract}
In this paper, we obtain some Hermite-Hadamard type inequalities for
functions whose third derivatives in absolute value are preinvex and
prequasiinvex.
\end{abstract}

\author{M.Emin Ozdemir$^{\blacklozenge }$}
\address{$^{\blacklozenge }$ATATURK UNIVERSITY, K.K. EDUCATION FACULTY,
DEPARTMENT OF MATHEMATICS, 25240 CAMPUS, ERZURUM, TURKEY}
\email{emos@atauni.edu.tr}
\author{Merve Avci Ardic$^{\ast ,\diamondsuit }$}
\thanks{$^{\diamondsuit }$Corresponding Author}
\address{$^{\ast }$ADIYAMAN UNIVERSITY, FACULTY OF SCIENCE AND ART,
DEPARTMENT OF MATHEMATICS, 02040, ADIYAMAN, TURKEY}
\email{mavci@adiyaman.edu.tr}
\maketitle

\section{INTRODUCTION}

Several researchers have been studied on convexity and a lot of papers have
been written on this topic which give new generalizations, extensions and
applications. A huge amount of these studies on refinements of celebrated
Hermite-Hadamard Inequality for convex functions. Invex functions introduced
by Hanson as a generalization of convex functions in \cite{5}. Some
properties of preinvex functions have been discussed in the papers \cite{9}-%
\cite{15}. It is well-known that there are many applications of invexity in
nonlinear optimization, variational inequalities and in the other branches
of pure applied sciences.

Now it is time to give the following definitions and results which will be
used in this paper (see \cite{11}, \cite{12} and \cite{13}):

Let $K$ be a nonempty closed set in $%
%TCIMACRO{\U{211d} }%
%BeginExpansion
\mathbb{R}
%EndExpansion
^{n}.$ We denote by $\langle .,.\rangle $ and $\left\Vert .\right\Vert $ the
inner product and norm respectively. Let $f:K\rightarrow 
%TCIMACRO{\U{211d} }%
%BeginExpansion
\mathbb{R}
%EndExpansion
$ and $\eta :K\times K\rightarrow 
%TCIMACRO{\U{211d} }%
%BeginExpansion
\mathbb{R}
%EndExpansion
$ be continuous functions.

\begin{definition}
(See \cite{12}) Let $u\in K.$ Then the set $K$ is said to be invex at $u$
with respect to $\eta \left( .,.\right) ,$ if 
\begin{equation*}
u+t\eta (v,u)\in K,\text{ \ \ }\forall u,v\in K,\text{ \ \ }t\in \lbrack
0,1].
\end{equation*}%
$K$ is said to be invex set with respect to $\eta ,$ if $K$ is invex at each 
$u\in K.$ The invex set $K$ is also called a $\eta -$connected set.
\end{definition}

\begin{remark}
(See \cite{11}) We would like to mention that the Definition 1 of an invex
set has a clear geometric interpretation. This definition essentially says
that there is a path starting from a point $u$ which is contained in $K.$ We
don't require that the point $v$ should be one of the end points of the
path. This observation plays an important role in our analysis . Note that,
if we demand that $v$ should be an end point of the path for every pair of
points, $u,v\in K,$ then $\eta (v,u)=v-u$ and consequently invexity reduces
to convexity. Thus, it is true that every convex set is also an invex set
with respect to $\eta (v,u)=v-u$, but the converse is not necessarily true.
\end{remark}

\begin{definition}
(See \cite{12}) The function $f$ on the invex set $K$ is said to be preinvex
with respect to $\eta ,$ if 
\begin{equation*}
f(u+t\eta (v,u))\leq (1-t)f(u)+tf(v),\text{ \ \ }\forall u,v\in K,\text{ \ \ 
}t\in \lbrack 0,1].
\end{equation*}%
The function $f$ is said to be preconcave if and only if $-f$ is preinvex.
Note that every convex function is a preinvex function, but the converse is
not true. For example, the function $f(u)=-\left\vert u\right\vert $ is not
a convex function, but it is a preinvex function with respect to $\eta ,$
where 
\begin{equation*}
\eta (v,u)=\left\{ 
\begin{array}{c}
v-u,\text{ \ \ \ \ if }v\leq 0,u\leq 0\text{ \ \ and \ \ }v\geq 0,u\geq 0 \\ 
u-v,\text{ \ \ \ \ \ \ \ \ \ \ \ \ \ \ \ \ \ \ \ \ \ \ \ \ \ \ \ \ \ \ \ \ \
\ \ \ \ \ \ \ \ \ \ otherwise}%
\end{array}%
\right. .
\end{equation*}
\end{definition}

\begin{definition}
(See \cite{10}) The function $f$ on the invex set $K$ is said to be
prequasiinvex with respect to $\eta ,$ if
\end{definition}

\begin{equation*}
f(u+t\eta (v,u))\leq \max \left\{ f(u),f(v)\right\} ,\text{ \ \ }\forall
u,v\in K,\text{ \ \ }t\in \lbrack 0,1].
\end{equation*}

The following inequality is well known in the literature as the
Hermite-Hadamard integral inequality:%
\begin{equation*}
f\left( \frac{a+b}{2}\right) \leq \frac{1}{b-a}\int_{a}^{b}f(x)dx\leq \frac{%
f(a)+f(b)}{2}
\end{equation*}%
where $f:I\subseteq 
%TCIMACRO{\U{211d} }%
%BeginExpansion
\mathbb{R}
%EndExpansion
\rightarrow 
%TCIMACRO{\U{211d} }%
%BeginExpansion
\mathbb{R}
%EndExpansion
$ is a convex function on the interval $I$ of real numbers and $a,b\in I$
with $a<b.$

\begin{lemma}
\label{lem 1.1} (See \cite{4}) Let $f:I\subseteq 
%TCIMACRO{\U{211d} }%
%BeginExpansion
\mathbb{R}
%EndExpansion
\rightarrow 
%TCIMACRO{\U{211d} }%
%BeginExpansion
\mathbb{R}
%EndExpansion
$ be a three times differentiable function on $I^{\circ }$ with $a,b\in I$, $%
a<b.$ If $f^{\prime \prime \prime }\in L[a,b],$ then%
\begin{eqnarray*}
&&\frac{f(a)+f(b)}{2}-\frac{1}{b-a}\int_{a}^{b}f(x)dx-\frac{b-a}{12}\left[
f^{\prime }(b)-f^{\prime }(a)\right] \\
&=&\frac{\left( b-a\right) ^{3}}{12}\int_{0}^{1}t\left( 1-t\right) \left(
2t-1\right) f^{\prime \prime \prime }\left( ta+(1-t)b\right) dt.
\end{eqnarray*}
\end{lemma}

The main purpose of this paper is to prove some new inequalities of
Hermite-Hadamard type for preinvex and prequasiinvex functions by using a
new version of Lemma \ref{lem 1.1}.

\section{Inequalities for preinvex functions}

To prove our main results we need the following equality which is a
generalization of Lemma 1 to invex sets:

\begin{lemma}
\label{lem 2.1} Let $A\subseteq 
%TCIMACRO{\U{211d} }%
%BeginExpansion
\mathbb{R}
%EndExpansion
$ be an open invex subset with respect to $\eta :A\times A\rightarrow 
%TCIMACRO{\U{211d} }%
%BeginExpansion
\mathbb{R}
%EndExpansion
$ and $a,b\in A$ with $\eta \left( b,a\right) \neq 0.$ Suppose that $%
f:A\rightarrow 
%TCIMACRO{\U{211d} }%
%BeginExpansion
\mathbb{R}
%EndExpansion
$ is a three times differentiable function. If $f^{\prime \prime \prime }$
is integrable on the $\eta -$path $P_{bc},$ $c=b+\eta \left( a,b\right) ,$
then the following equality holds:%
\begin{eqnarray}
&&  \label{1} \\
&&\int_{b}^{b+\eta \left( a,b\right) }f(x)dx-\eta \left( a,b\right) \frac{%
f(b)+f(b+\eta \left( a,b\right) )}{2}-\frac{\left( \eta \left( a,b\right)
\right) ^{2}}{12}\left[ f^{\prime }(b)-f^{\prime }(b+\eta \left( a,b\right) )%
\right]  \notag \\
&=&\frac{\left( \eta \left( a,b\right) \right) ^{4}}{12}\int_{0}^{1}t\left(
1-t\right) \left( 2t-1\right) f^{\prime \prime \prime }(b+t\eta \left(
a,b\right) )dt.  \notag
\end{eqnarray}
\end{lemma}

\begin{proof}
The proof of Lemma \ref{lem 2.1} is left to the reader.
\end{proof}

\begin{theorem}
\label{teo 2.1} Let $A\subseteq 
%TCIMACRO{\U{211d} }%
%BeginExpansion
\mathbb{R}
%EndExpansion
$ be an open invex subset with respect to $\eta :A\times A\rightarrow 
%TCIMACRO{\U{211d} }%
%BeginExpansion
\mathbb{R}
%EndExpansion
.$ Suppose that $f:A\rightarrow 
%TCIMACRO{\U{211d} }%
%BeginExpansion
\mathbb{R}
%EndExpansion
$ is a three times differentiable function. If $\left\vert f^{\prime \prime
\prime }\right\vert ^{q}$ is preinvex on $A$ $,$ then the following
inequality holds:%
\begin{eqnarray*}
&&\left\vert \int_{b}^{b+\eta \left( a,b\right) }f(x)dx-\eta \left(
a,b\right) \frac{f(b)+f(b+\eta \left( a,b\right) )}{2}-\frac{\left( \eta
\left( a,b\right) \right) ^{2}}{12}\left[ f^{\prime }(b)-f^{\prime }(b+\eta
\left( a,b\right) )\right] \right\vert \\
&\leq &\frac{\left( \eta \left( a,b\right) \right) ^{4}}{192}\left[ \frac{%
\left\vert f^{\prime \prime \prime }(a)\right\vert ^{q}+\left\vert f^{\prime
\prime \prime }(b)\right\vert ^{q}}{2}\right] ^{\frac{1}{q}}
\end{eqnarray*}%
for $q\geq 1$ and every $a,b\in A$ with $\eta \left( b,a\right) \neq 0.$
\end{theorem}

\begin{proof}
Since $a,b\in A$ and $A$ is an invex set with respect to $\eta ,$ it is
obvious that $b+t\eta \left( a,b\right) \in A$ for $t\in \left[ 0,1\right] .$
By using Lemma \ref{lem 2.1}, the power-mean inequality and preinvexity of $%
\left\vert f^{\prime \prime \prime }\right\vert ^{q}$, we can write%
\begin{eqnarray*}
&&\left\vert \int_{b}^{b+\eta \left( a,b\right) }f(x)dx-\eta \left(
a,b\right) \frac{f(b)+f(b+\eta \left( a,b\right) )}{2}-\frac{\left( \eta
\left( a,b\right) \right) ^{2}}{12}\left[ f^{\prime }(b)-f^{\prime }(b+\eta
\left( a,b\right) )\right] \right\vert \\
&=&\frac{\left( \eta \left( a,b\right) \right) ^{4}}{12}\int_{0}^{1}t\left(
1-t\right) \left\vert 2t-1\right\vert \left\vert f^{\prime \prime \prime
}(b+t\eta \left( a,b\right) )\right\vert \\
&\leq &\frac{\left( \eta \left( a,b\right) \right) ^{4}}{12}\left(
\int_{0}^{1}t\left( 1-t\right) \left\vert 2t-1\right\vert dt\right) ^{1-%
\frac{1}{q}}\left( \int_{0}^{1}t\left( 1-t\right) \left\vert 2t-1\right\vert
\left\vert f^{\prime \prime \prime }(b+t\eta \left( a,b\right) )\right\vert
^{q}dt\right) ^{\frac{1}{q}} \\
&\leq &\frac{\left( \eta \left( a,b\right) \right) ^{4}}{12}\left(
\int_{0}^{1}t\left( 1-t\right) \left\vert 2t-1\right\vert dt\right) ^{1-%
\frac{1}{q}}\left( \int_{0}^{1}t\left( 1-t\right) \left\vert 2t-1\right\vert %
\left[ t\left\vert f^{\prime \prime \prime }(a)\right\vert
^{q}+(1-t)\left\vert f^{\prime \prime \prime }(b)\right\vert ^{q}\right]
dt\right) ^{\frac{1}{q}} \\
&=&\frac{\left( \eta \left( a,b\right) \right) ^{4}}{192}\left[ \frac{%
\left\vert f^{\prime \prime \prime }(a)\right\vert ^{q}+\left\vert f^{\prime
\prime \prime }(b)\right\vert ^{q}}{2}\right] ^{\frac{1}{q}}
\end{eqnarray*}%
where we use the fact that%
\begin{equation*}
\int_{0}^{1}t\left( 1-t\right) \left\vert 2t-1\right\vert dt=\frac{1}{16}
\end{equation*}%
and%
\begin{equation*}
\int_{0}^{1}t^{2}\left( 1-t\right) \left\vert 2t-1\right\vert
dt=\int_{0}^{1}t\left( 1-t\right) ^{2}\left\vert 2t-1\right\vert dt=\frac{1}{%
32}.
\end{equation*}%
The proof is completed.
\end{proof}

\begin{corollary}
\label{co 2.1} In Theorem \ref{teo 2.1}, if we choose $q=1,$ we obtain%
\begin{eqnarray*}
&&\left\vert \int_{b}^{b+\eta \left( a,b\right) }f(x)dx-\eta \left(
a,b\right) \frac{f(b)+f(b+\eta \left( a,b\right) )}{2}-\frac{\left( \eta
\left( a,b\right) \right) ^{2}}{12}\left[ f^{\prime }(b)-f^{\prime }(b+\eta
\left( a,b\right) )\right] \right\vert \\
&\leq &\frac{\left( \eta \left( a,b\right) \right) ^{4}}{384}\left[
\left\vert f^{\prime \prime \prime }(a)\right\vert +\left\vert f^{\prime
\prime \prime }(b)\right\vert \right] .
\end{eqnarray*}
\end{corollary}

\begin{corollary}
\label{co 2.2} In Theorem \ref{teo 2.1}, if we take $f^{\prime
}(b)=f^{\prime }(b+\eta \left( a,b\right) ),$ then we have%
\begin{eqnarray*}
&&\left\vert \int_{b}^{b+\eta \left( a,b\right) }f(x)dx-\eta \left(
a,b\right) \frac{f(b)+f(b+\eta \left( a,b\right) )}{2}\right\vert \\
&\leq &\frac{\left( \eta \left( a,b\right) \right) ^{4}}{384}\left[
\left\vert f^{\prime \prime \prime }(a)\right\vert ^{q}+\left\vert f^{\prime
\prime \prime }(b)\right\vert ^{q}\right] ^{\frac{1}{q}}.
\end{eqnarray*}
\end{corollary}

\begin{theorem}
\label{teo 2.2} Let $A\subseteq 
%TCIMACRO{\U{211d} }%
%BeginExpansion
\mathbb{R}
%EndExpansion
$ be an open invex subset with respect to $\eta :A\times A\rightarrow 
%TCIMACRO{\U{211d} }%
%BeginExpansion
\mathbb{R}
%EndExpansion
.$ Suppose that $f:A\rightarrow 
%TCIMACRO{\U{211d} }%
%BeginExpansion
\mathbb{R}
%EndExpansion
$ is a differentiable function. If $\left\vert f^{\prime \prime }\right\vert
^{q}$ is preinvex on $A,$ then the following inequality holds:%
\begin{eqnarray*}
&&\left\vert \int_{b}^{b+\eta \left( a,b\right) }f(x)dx-\eta \left(
a,b\right) \frac{f(b)+f(b+\eta \left( a,b\right) )}{2}-\frac{\left( \eta
\left( a,b\right) \right) ^{2}}{12}\left[ f^{\prime }(b)-f^{\prime }(b+\eta
\left( a,b\right) )\right] \right\vert \\
&\leq &\frac{\left( \eta \left( b,a\right) \right) ^{4}}{24\times 6^{\frac{1%
}{q}}}\left( \frac{1}{\left( p+1\right) \left( p+3\right) }\right) ^{\frac{1%
}{p}}\left[ \left\vert f^{\prime \prime \prime }(a)\right\vert
^{q}+\left\vert f^{\prime \prime \prime }(b)\right\vert ^{q}\right] ^{\frac{1%
}{q}}
\end{eqnarray*}%
for every $a,b\in A$ with $\eta \left( b,a\right) \neq 0$ where $q>1$, $%
p^{-1}+q^{-1}=1$.
\end{theorem}

\begin{proof}
By using Lemma \ref{lem 2.1}, H\"{o}lder inequality and preinvexity of $%
\left\vert f^{\prime \prime \prime }\right\vert ^{q}$, we can write%
\begin{eqnarray*}
&&\left\vert \int_{b}^{b+\eta \left( a,b\right) }f(x)dx-\eta \left(
a,b\right) \frac{f(b)+f(b+\eta \left( a,b\right) )}{2}-\frac{\left( \eta
\left( a,b\right) \right) ^{2}}{12}\left[ f^{\prime }(b)-f^{\prime }(b+\eta
\left( a,b\right) )\right] \right\vert \\
&\leq &\frac{\left( \eta \left( b,a\right) \right) ^{4}}{12}\left(
\int_{0}^{1}t\left( 1-t\right) \left\vert 2t-1\right\vert ^{p}dt\right) ^{%
\frac{1}{p}}\left( \int_{0}^{1}t\left( 1-t\right) \left\vert f^{\prime
\prime }(b+t\eta \left( a,b\right) )\right\vert ^{q}dt\right) ^{\frac{1}{q}}
\\
&\leq &\frac{\left( \eta \left( b,a\right) \right) ^{4}}{12}\left(
\int_{0}^{1}t\left( 1-t\right) \left\vert 2t-1\right\vert ^{p}dt\right) ^{%
\frac{1}{p}}\left( \int_{0}^{1}t\left( 1-t\right) \left[ t\left\vert
f^{\prime \prime \prime }(a)\right\vert ^{q}+(1-t)\left\vert f^{\prime
\prime \prime }(b)\right\vert ^{q}\right] dt\right) ^{\frac{1}{q}}.
\end{eqnarray*}%
Computing the above integrals, we deduce%
\begin{eqnarray*}
&&\left\vert \int_{b}^{b+\eta \left( a,b\right) }f(x)dx-\eta \left(
a,b\right) \frac{f(b)+f(b+\eta \left( a,b\right) )}{2}-\frac{\left( \eta
\left( a,b\right) \right) ^{2}}{12}\left[ f^{\prime }(b)-f^{\prime }(b+\eta
\left( a,b\right) )\right] \right\vert \\
&\leq &\frac{\left( \eta \left( b,a\right) \right) ^{4}}{24\times 6^{\frac{1%
}{q}}}\left( \frac{1}{\left( p+1\right) \left( p+3\right) }\right) ^{\frac{1%
}{p}}\left[ \left\vert f^{\prime \prime \prime }(a)\right\vert
^{q}+\left\vert f^{\prime \prime \prime }(b)\right\vert ^{q}\right] ^{\frac{1%
}{q}},
\end{eqnarray*}%
which completes the proof.
\end{proof}

\begin{theorem}
\label{teo 2.3} Under the assumptions of Theorem \ref{teo 2.2}, we have%
\begin{eqnarray*}
&&\left\vert \int_{b}^{b+\eta \left( a,b\right) }f(x)dx-\eta \left(
a,b\right) \frac{f(b)+f(b+\eta \left( a,b\right) )}{2}-\frac{\left( \eta
\left( a,b\right) \right) ^{2}}{12}\left[ f^{\prime }(b)-f^{\prime }(b+\eta
\left( a,b\right) )\right] \right\vert \\
&\leq &\frac{\left( \eta \left( b,a\right) \right) ^{4}}{96}\left( \sqrt{\pi 
}\right) ^{\frac{1}{p}}\left( \frac{\Gamma \left( 1+p\right) }{\Gamma \left( 
\frac{3}{2}+p\right) }\right) ^{\frac{1}{p}}\left( \frac{1}{q+1}\right) ^{%
\frac{1}{q}}\left[ \left\vert f^{\prime \prime \prime }(a)\right\vert
^{q}+\left\vert f^{\prime \prime \prime }(b)\right\vert ^{q}\right] ^{\frac{1%
}{q}}.
\end{eqnarray*}
\end{theorem}

\begin{proof}
From Lemma \ref{lem 2.1}, using H\"{o}lder inequality and preinvexity of $%
\left\vert f^{\prime \prime }\right\vert ^{q},$ we have%
\begin{eqnarray*}
&&\left\vert \int_{b}^{b+\eta \left( a,b\right) }f(x)dx-\eta \left(
a,b\right) \frac{f(b)+f(b+\eta \left( a,b\right) )}{2}-\frac{\left( \eta
\left( a,b\right) \right) ^{2}}{12}\left[ f^{\prime }(b)-f^{\prime }(b+\eta
\left( a,b\right) )\right] \right\vert \\
&\leq &\frac{\left( \eta \left( b,a\right) \right) ^{4}}{12}\left(
\int_{0}^{1}\left( t-t^{2}\right) ^{p}dt\right) ^{\frac{1}{p}}\left(
\int_{0}^{1}\left\vert 2t-1\right\vert ^{q}\left\vert f^{\prime \prime
\prime }(b+t\eta \left( a,b\right) )\right\vert ^{q}dt\right) ^{\frac{1}{q}}
\\
&\leq &\frac{\left( \eta \left( b,a\right) \right) ^{4}}{12}\left(
\int_{0}^{1}\left( t-t^{2}\right) ^{p}dt\right) ^{\frac{1}{p}}\left(
\int_{0}^{1}\left\vert 2t-1\right\vert ^{q}\left[ t\left\vert f^{\prime
\prime }(a)\right\vert ^{q}+(1-t)\left\vert f^{\prime \prime }(b)\right\vert
^{q}\right] dt\right) ^{\frac{1}{q}} \\
&=&\frac{\left( \eta \left( b,a\right) \right) ^{4}}{12}\frac{\left( \sqrt{%
\pi }\right) ^{\frac{1}{p}}}{2^{\frac{1}{p}}\times 4}\left( \frac{\Gamma
\left( 1+p\right) }{\Gamma \left( \frac{3}{2}+p\right) }\right) ^{\frac{1}{p}%
}\left[ \frac{\left\vert f^{\prime \prime \prime }(a)\right\vert
^{q}+\left\vert f^{\prime \prime \prime }(b)\right\vert ^{q}}{2\left(
q+1\right) }\right] ^{\frac{1}{q}}
\end{eqnarray*}%
where we used the fact that%
\begin{equation*}
\int_{0}^{1}\left( t-t^{2}\right) ^{p}dt=\frac{2^{-1-2p}\sqrt{\pi }\Gamma
\left( 1+p\right) }{\Gamma \left( \frac{3}{2}+p\right) }
\end{equation*}%
and%
\begin{equation*}
\int_{0}^{1}t\left\vert 2t-1\right\vert ^{q}dt=\int_{0}^{1}\left( 1-t\right)
\left\vert 2t-1\right\vert ^{q}dt=\frac{1}{2\left( q+1\right) }.
\end{equation*}%
The proof is completed.
\end{proof}

\section{Inequalities for prequasiinvex functions}

In this section, we obtain Hermite-Hadamard type inequalities for
prequasiinvex functions via Lemma \ref{lem 2.1}.

\begin{theorem}
\label{teo 3.1} Let $A\subseteq 
%TCIMACRO{\U{211d} }%
%BeginExpansion
\mathbb{R}
%EndExpansion
$ be an open invex subset with respect to $\eta :A\times A\rightarrow 
%TCIMACRO{\U{211d} }%
%BeginExpansion
\mathbb{R}
%EndExpansion
.$ Suppose that $f:A\rightarrow 
%TCIMACRO{\U{211d} }%
%BeginExpansion
\mathbb{R}
%EndExpansion
$ is a three times differentiable function. If $\left\vert f^{\prime \prime
\prime }\right\vert ^{q}$ is prequasiinvex on $A$ $,$ then the following
inequality holds:%
\begin{eqnarray*}
&&\left\vert \int_{b}^{b+\eta \left( a,b\right) }f(x)dx-\eta \left(
a,b\right) \frac{f(b)+f(b+\eta \left( a,b\right) )}{2}-\frac{\left( \eta
\left( a,b\right) \right) ^{2}}{12}\left[ f^{\prime }(b)-f^{\prime }(b+\eta
\left( a,b\right) )\right] \right\vert \\
&\leq &\frac{\left( \eta \left( a,b\right) \right) ^{4}}{192}\left[ \max
\left\{ \left\vert f^{\prime \prime \prime }(a)\right\vert ^{q},\left\vert
f^{\prime \prime \prime }(b)\right\vert ^{q}\right\} \right] ^{\frac{1}{q}}
\end{eqnarray*}%
for $q\geq 1$ and every $a,b\in A$ with $\eta \left( b,a\right) \neq 0.$
\end{theorem}

\begin{proof}
By using Lemma \ref{lem 2.1}, the power-mean inequality and prequasiinvexity
of $\left\vert f^{\prime \prime \prime }\right\vert ^{q}$, we can write%
\begin{eqnarray*}
&&\left\vert \int_{b}^{b+\eta \left( a,b\right) }f(x)dx-\eta \left(
a,b\right) \frac{f(b)+f(b+\eta \left( a,b\right) )}{2}-\frac{\left( \eta
\left( a,b\right) \right) ^{2}}{12}\left[ f^{\prime }(b)-f^{\prime }(b+\eta
\left( a,b\right) )\right] \right\vert \\
&\leq &\frac{\left( \eta \left( a,b\right) \right) ^{4}}{12}\left(
\int_{0}^{1}t\left( 1-t\right) \left\vert 2t-1\right\vert dt\right) ^{1-%
\frac{1}{q}}\left( \int_{0}^{1}t\left( 1-t\right) \left\vert 2t-1\right\vert
\left\vert f^{\prime \prime \prime }(b+t\eta \left( a,b\right) )\right\vert
^{q}dt\right) ^{\frac{1}{q}} \\
&\leq &\frac{\left( \eta \left( a,b\right) \right) ^{4}}{12}\left(
\int_{0}^{1}t\left( 1-t\right) \left\vert 2t-1\right\vert dt\right) ^{1-%
\frac{1}{q}}\left( \int_{0}^{1}t\left( 1-t\right) \left\vert 2t-1\right\vert %
\left[ \max \left\{ \left\vert f^{\prime \prime \prime }(a)\right\vert
^{q},\left\vert f^{\prime \prime \prime }(b)\right\vert ^{q}\right\} \right]
dt\right) ^{\frac{1}{q}} \\
&=&\frac{\left( \eta \left( a,b\right) \right) ^{4}}{12}\left( \frac{1}{16}%
\right) ^{1-\frac{1}{q}}\left[ \frac{\max \left\{ \left\vert f^{\prime
\prime \prime }(a)\right\vert ^{q},\left\vert f^{\prime \prime \prime
}(b)\right\vert ^{q}\right\} }{16}\right] ^{\frac{1}{q}}.
\end{eqnarray*}%
The proof is completed.
\end{proof}

\begin{corollary}
\label{co 2.3} In Theorem \ref{teo 3.1}, if we choose $q=1,$ we obtain%
\begin{eqnarray*}
&&\left\vert \int_{b}^{b+\eta \left( a,b\right) }f(x)dx-\eta \left(
a,b\right) \frac{f(b)+f(b+\eta \left( a,b\right) )}{2}-\frac{\left( \eta
\left( a,b\right) \right) ^{2}}{12}\left[ f^{\prime }(b)-f^{\prime }(b+\eta
\left( a,b\right) )\right] \right\vert \\
&\leq &\frac{\left( \eta \left( a,b\right) \right) ^{4}}{192}\left[ \max
\left\{ \left\vert f^{\prime \prime \prime }(a)\right\vert ,\left\vert
f^{\prime \prime \prime }(b)\right\vert \right\} \right] .
\end{eqnarray*}
\end{corollary}

\begin{corollary}
\label{co 2.4} In Theorem \ref{teo 3.1}, if we take $f^{\prime
}(b)=f^{\prime }(b+\eta \left( a,b\right) ),$ then we have%
\begin{eqnarray*}
&&\left\vert \int_{b}^{b+\eta \left( a,b\right) }f(x)dx-\eta \left(
a,b\right) \frac{f(b)+f(b+\eta \left( a,b\right) )}{2}\right\vert  \\
&\leq &\frac{\left( \eta \left( a,b\right) \right) ^{4}}{192}\left[ \max
\left\{ \left\vert f^{\prime \prime \prime }(a)\right\vert ^{q},\left\vert
f^{\prime \prime \prime }(b)\right\vert ^{q}\right\} \right] ^{\frac{1}{q}}.
\end{eqnarray*}
\end{corollary}

\begin{theorem}
\label{teo 3.2} Let $A\subseteq 
%TCIMACRO{\U{211d} }%
%BeginExpansion
\mathbb{R}
%EndExpansion
$ be an open invex subset with respect to $\eta :A\times A\rightarrow 
%TCIMACRO{\U{211d} }%
%BeginExpansion
\mathbb{R}
%EndExpansion
.$ Suppose that $f:A\rightarrow 
%TCIMACRO{\U{211d} }%
%BeginExpansion
\mathbb{R}
%EndExpansion
$ is a differentiable function. If $\left\vert f^{\prime \prime }\right\vert
^{q}$ is preinvex on $A,$ then the following inequality holds:%
\begin{eqnarray*}
&&\left\vert \int_{b}^{b+\eta \left( a,b\right) }f(x)dx-\eta \left(
a,b\right) \frac{f(b)+f(b+\eta \left( a,b\right) )}{2}-\frac{\left( \eta
\left( a,b\right) \right) ^{2}}{12}\left[ f^{\prime }(b)-f^{\prime }(b+\eta
\left( a,b\right) )\right] \right\vert \\
&\leq &\frac{\left( \eta \left( b,a\right) \right) ^{4}}{24\times 3^{\frac{1%
}{q}}}\left( \frac{1}{\left( p+1\right) \left( p+3\right) }\right) ^{\frac{1%
}{p}}\left[ \max \left\{ \left\vert f^{\prime \prime \prime }(a)\right\vert
^{q},\left\vert f^{\prime \prime \prime }(b)\right\vert ^{q}\right\} \right]
^{\frac{1}{q}}
\end{eqnarray*}%
for every $a,b\in A$ with $\eta \left( b,a\right) \neq 0$ where $q>1$, $%
p^{-1}+q^{-1}=1$.
\end{theorem}

\begin{proof}
By using Lemma \ref{lem 2.1}, H\"{o}lder inequality and preinvexity of $%
\left\vert f^{\prime \prime \prime }\right\vert ^{q}$, we can write%
\begin{eqnarray*}
&&\left\vert \int_{b}^{b+\eta \left( a,b\right) }f(x)dx-\eta \left(
a,b\right) \frac{f(b)+f(b+\eta \left( a,b\right) )}{2}-\frac{\left( \eta
\left( a,b\right) \right) ^{2}}{12}\left[ f^{\prime }(b)-f^{\prime }(b+\eta
\left( a,b\right) )\right] \right\vert \\
&\leq &\frac{\left( \eta \left( b,a\right) \right) ^{4}}{12}\left(
\int_{0}^{1}t\left( 1-t\right) \left\vert 2t-1\right\vert ^{p}dt\right) ^{%
\frac{1}{p}}\left( \int_{0}^{1}t\left( 1-t\right) \left\vert f^{\prime
\prime }(b+t\eta \left( a,b\right) )\right\vert ^{q}dt\right) ^{\frac{1}{q}}
\\
&\leq &\frac{\left( \eta \left( b,a\right) \right) ^{4}}{12}\left(
\int_{0}^{1}t\left( 1-t\right) \left\vert 2t-1\right\vert ^{p}dt\right) ^{%
\frac{1}{p}}\left( \int_{0}^{1}t\left( 1-t\right) \left[ \max \left\{
\left\vert f^{\prime \prime \prime }(a)\right\vert ^{q},\left\vert f^{\prime
\prime \prime }(b)\right\vert ^{q}\right\} \right] dt\right) ^{\frac{1}{q}}.
\end{eqnarray*}%
Computing the above integrals, we obtain the desired result.
\end{proof}

\begin{theorem}
\label{teo 3.3} Under the assumptions of Theorem \ref{teo 3.2}, we have%
\begin{eqnarray*}
&&\left\vert \int_{b}^{b+\eta \left( a,b\right) }f(x)dx-\eta \left(
a,b\right) \frac{f(b)+f(b+\eta \left( a,b\right) )}{2}-\frac{\left( \eta
\left( a,b\right) \right) ^{2}}{12}\left[ f^{\prime }(b)-f^{\prime }(b+\eta
\left( a,b\right) )\right] \right\vert \\
&\leq &\frac{\left( \eta \left( b,a\right) \right) ^{4}}{48}\left( \sqrt{\pi 
}\right) ^{\frac{1}{p}}\left( \frac{\Gamma \left( 1+p\right) }{\Gamma \left( 
\frac{3}{2}+p\right) }\right) ^{\frac{1}{p}}\left( \frac{1}{q+1}\right) ^{%
\frac{1}{q}}\left[ \max \left\{ \left\vert f^{\prime \prime \prime
}(a)\right\vert ^{q},\left\vert f^{\prime \prime \prime }(b)\right\vert
^{q}\right\} \right] ^{\frac{1}{q}}.
\end{eqnarray*}
\end{theorem}

\begin{proof}
From Lemma \ref{lem 2.1}, using H\"{o}lder inequality and preinvexity of $%
\left\vert f^{\prime \prime }\right\vert ^{q},$ we have%
\begin{eqnarray*}
&&\left\vert \int_{b}^{b+\eta \left( a,b\right) }f(x)dx-\eta \left(
a,b\right) \frac{f(b)+f(b+\eta \left( a,b\right) )}{2}-\frac{\left( \eta
\left( a,b\right) \right) ^{2}}{12}\left[ f^{\prime }(b)-f^{\prime }(b+\eta
\left( a,b\right) )\right] \right\vert \\
&\leq &\frac{\left( \eta \left( b,a\right) \right) ^{4}}{12}\left(
\int_{0}^{1}\left( t-t^{2}\right) ^{p}dt\right) ^{\frac{1}{p}}\left(
\int_{0}^{1}\left\vert 2t-1\right\vert ^{q}\left\vert f^{\prime \prime
\prime }(b+t\eta \left( a,b\right) )\right\vert ^{q}dt\right) ^{\frac{1}{q}}
\\
&\leq &\frac{\left( \eta \left( b,a\right) \right) ^{4}}{12}\left(
\int_{0}^{1}\left( t-t^{2}\right) ^{p}dt\right) ^{\frac{1}{p}}\left(
\int_{0}^{1}\left\vert 2t-1\right\vert ^{q}\left[ \max \left\{ \left\vert
f^{\prime \prime \prime }(a)\right\vert ^{q},\left\vert f^{\prime \prime
\prime }(b)\right\vert ^{q}\right\} \right] dt\right) ^{\frac{1}{q}}.
\end{eqnarray*}%
Computing the above integrals, we obtain the desired result.
\end{proof}

\end{document}